\documentclass[11pt,twoside]{article}
\usepackage[english]{babel}
\usepackage[T1]{fontenc}
\usepackage[latin1]{inputenc}
\usepackage{amsmath,amssymb,mathrsfs,version,graphicx,subfig,float,caption,amsthm}
\usepackage[left=3cm, right=2cm, top=4cm, bottom=5cm]{geometry}
\usepackage{dsfont}
\usepackage{bbold}
\usepackage{enumitem}

	\newcommand{\R}{{\mathbb R}}
	\newcommand{\N}{{\mathbb N}}
	
	\newtheorem*{defi}{Definition}
	
	\newtheorem{lem}{Lemma}
	
	\newtheorem{prop}{Proposition}
	\newtheorem{cor}{Corollary}
	\newtheorem{theo}{Theorem}
	
	\theoremstyle{remark}
	\newtheorem{rem}{Remark}

\begin{document}
\title{Perimeters, uniform enlargement and high dimensions
}
\author{F. Barthe\thanks{Partially supported by ANR 2011 BS01 007 01 (GeMeCoD project)} \, and B. Huou
\\ Universit\'e de Toulouse, CNRS}
\maketitle

\begin{abstract}
We study the isoperimetric problem in product spaces equipped with the uniform distance.
Our main result is a characterization of isoperimetric inequalities which, when satisfied on a space, are still valid for the product spaces, up a to a constant which does not depend 
on the number of factors. Such dimension free bounds have applications to the study of
influences of variables.
\end{abstract}

\section{Introduction}
 Let $(X,d,\mu)$ denote a metric probability space, where $X$ is separable and $\mu$ is a Borel probability measure on $(X,d)$. For a Borel subset $A$ of $X$, we define, for $r > 0$, the open r-neighbourhood of $A$ by $A_r = \left\{x \in X \,\big| \, d(x,A) < r \right\}$, and its outer and inner boundary measures (also called Minkowski contents) by

\begin{equation*}
\mu^+(A) = \underset{r \rightarrow 0^+}{\liminf}~\frac{\mu\left(A_r\right)-\mu\left(A\right)}{r},\quad \mu^-(A)=\mu^+(X\setminus A).
\end{equation*}
The isoperimetric problem consists in obtaining sharp lower bounds on the above quantities
in terms of the measure $\mu(A)$.
The isoperimetric function of $(X,d,\mu)$, denoted by $I_{(X,d,\mu)}$ (or simply  $I_{\mu}$ 
when there is no ambiguity on the underlying metric space), is defined for $p \in \left[0,1\right]$ as follows :

\begin{eqnarray}\label{a}
I_{\mu}(p) &=&\underset{A \subseteq X ;~\mu(A)=p}{\inf}\min \big(\mu^+(A),\mu^-(A)\big)\\
&=&\underset{A \subseteq X ;~\mu(A)\in \{p,1-p\}}{\inf} \mu^+(A)
\end{eqnarray}
where the infimum is taken over all Borel subsets $A$ of $X$. As we can see from the definition, $I_{\mu}$ is the largest function such that, for every $A \subseteq X$, $\mu^+(A) \geq I_{\mu}\left(\mu(A)\right)$ and for every $t \in \left[0,1\right]$, $I_{\mu}(t) = I_{\mu}(1-t)$. Notice also that $I_{\mu}(0) = I_{\mu}(1) = 0$.

\medskip

\noindent 
Given metric probability spaces $(X_i, d_i,\mu_i)$, $i=1,\ldots, n$, several metric 
structures can be considered on the product probability space $(X_1\times \cdots\times X_n,
\mu_1\otimes\cdots\otimes \mu_n)$. Throughout this paper, we equip this product with 
the supremum distance $d=d_\infty^{(n)}$ defined by
$$ d_\infty^{(n)}\big( (x_1,\ldots,x_n), (y_1,\ldots,y_n)\big):=\max_i d_i(x_i,y_i).$$
We shall also say that $d_\infty^{(n)}$ is the $\ell_\infty$-combination of the distances
$d_i$, $1\le i\le n$.
The isoperimetric problem has been intensively studied in the Riemannian setting, where
the geodesic distance on a product manifold is the $\ell_2$-combination of the geodesic 
distance on the factors. Hence, from a geometric viewpoint, the choice of the $\ell_\infty$-combination is less natural than the one of the $\ell_2$-combination 
$d_2^{(n)}((x_i)_{i=1}^n, (y_i)_{i=1}^n)=\big(\sum_i d(x_i,y_i)^2\big)^{1/2}$. Nevertheless,
the study of the uniform enlargement has various motivations. We briefly explain some of them. 

    Firstly the isoperimetric problem for the uniform enlargement is technically easier  to deal with in the setting of product spaces, due to the product structure of metric balls.
    This often allows to work by comparisons. For instance Bollob\'as and Leader study this problem for the uniform measure on the cube in order to solve the discrete isoperimetric problem on the grid \cite{bolll91eiig}. Since $d_\infty^{(n)}\le d_2^{(n)}\le \sqrt{n}\, d_\infty^{(n)}$, it easily follows that 
                           $$ \frac{1}{\sqrt n}\,   I_{(X^n, d_\infty^{(n)}, \mu^n)}\le  I_{(X^n,d_2^{(n)}, \mu^n )}\le I_{(X^n, d_\infty^{(n)}, \mu^n)}.$$
This approach was used e.g. by Morgan \cite{morg06iep} for products of two Riemannian manifolds.

     Another motivation for studying the isoperimetric problem for the uniform enlargement
     is that it amounts to the study of the usual isoperimetric problem for a special class of sets.
      Let us explain this briefly in the setting of $\R^n$ equipped with a probability  
      measure $d\mu(x)=\rho(x) dx$ and the $\ell_\infty$ distance.  If $\rho$ is continuous
       and $A\subset \R^n$ is a domain with Lipschitz boundary, its outer Minkowski content is 
        $$ \mu^+(A)= \int_{\partial A} \| n_A(x)\|_1 \rho(x) d\mathcal H_{n-1}(x),$$
        where $n_A(x)$ is a  unit outer normal to $A$ at $x$ (unit for the Euclidean length).
      Consequently, the boundary measure for the uniform enlargement coincides
       with the usual one $   \int_{\partial A} \rho(x) d\mathcal H_{n-1}(x),$ for sets
    $A$ such that almost surely on $\partial A$ the outer normal is equal to a vector 
     of the canonical basis of $\R^n$ (or its opposite). These so-called rectilinear sets
      comprise cartesian products of intervals $I_1\times\cdots \times I_n$, their finite unions
       and their complements. Hence the isoperimetric problem for the uniform enlargement
       is closely connected to the usual isoperimetric problem restricted to the class of 
        rectilinear sets (actually, a smooth domain $A$ can be approximated by rectilinear
         sets in such a way that their boundary measures approach the one of $A$ for the
         uniform enlargement). Note that rectilinear sets naturally appear when studying
          the supremum of random variables, as $\{x\in\R^n\,\Big| \, \max_i x_i \in [a,b]\}$ is rectilinear.
         This was one of the original motivations of Bobkov and Bobkov-Houdr\'e \cite{bobk97ipue,bobkh00wdfc} for studying isoperimetry for
         the uniform enlargement. 
          
          Eventually, let us mention that isoperimetric inequalities for the uniform enlargement
          naturally appear in the recent extension by Keller, Mossel and Sen \cite{kellms12gi}
          of the theory of influences of variables to the continuous setting.

\medskip

Computing exactly the isoperimetric profile is a hard task, even in simple product spaces (see e.g. the survey article \cite{ros01ip}). However, various probabilistic questions involve sequences of independent random variables and require lower estimates on the isoperimetric profile of $n$-fold product spaces, which actually do not depend on the value of $n$.  First observe that for all integers $n\ge 1$,
  $$I_{(X^{n+1},~d_\infty^{(n+1)},~\mu^{n+1})}\le I_{(X^{n},~d_\infty^{(n)},~\mu^{n})},$$
  which holds because for every set $A\subset X^n$, $\mu^{n+1}(A\times X)=\mu^n(A)$
  and $(\mu^{n+1})^+(A\times X)=(\mu^n)^+(A)$. Therefore one may define the so-called
  infinite dimensional isoperimetric profile of $(X,d,\mu)$ as follows~: for $t\in[0,1]$,
  $$I_{\mu^\infty}(t):=\inf_{n\ge 1}  I_{(X^{n},~d_\infty^{(n)},~\mu^{n})}\le I_{(X,d,\mu)} .$$
  This quantity has been investigated by Bobkov \cite{bobk97ipue}, Bobkov and Houdr\'e \cite{bobkh00wdfc} and Barthe \cite{bart04idii}.
In particular, Bobkov has put forward a sufficient condition for the equality $I_{\mu^\infty}=I_\mu$ to hold.  This condition depends only on 
the function $I_\mu$ but it is rather restrictive. However it allowed to get a natural family  
of isoperimetric inequalities for which there exists $K>1$ such that 
 $I_\mu\ge I_{\mu^\infty}\ge \frac{1}{K} I_{\mu}$. We shall say in this case that the isoperimetric inequality with profile $I_\mu$
 tensorizes, up to a factor $K$.

The goal of this article is to provide a workable necessary and sufficient condition for the 
latter property to hold. We were inspired by a sufficient condition for tensorization, given by E. Milman \cite{milm09rcfi}
in the setting of $\ell_2$-distances on products. We now describe the plan of the paper.
 In the next section, we recall the known sufficient condition for $I_{\mu^\infty}=I_\mu$ and propose a new one. Building on this,
 we provide a sufficient condition for tensorization up to a factor in the third section. By a careful study of product sets, we actually
 show that this condition is also necessary. The final section draws consequences of our isoperimetric inequalities to the theory 
 of influences of variables~: following the argument of \cite{kellms12gi}, we obtain an extension of the Kahn-Kalai-Linial 
 theorem about the existence of a coordinate with a large influence.

\smallskip
Let us conclude this introduction with   some useful notation.
If $(Y,\rho)$ is a metric space we define the modulus of gradient of a locally Lipschitz function $f : Y \rightarrow \R$ by :
\begin{equation*}
\vert\nabla f\vert (x) = \underset{\rho(x,y) \rightarrow 0^+}{\limsup}\frac{\vert f(x)-f(y)\vert}{\rho(x,y)},
\end{equation*}
this quantity being zero at isolated points. Note that when the distance is given by a norm
on a vector space, that is $\rho(x,y)=\|x-y\|$, and when $f$ is differentiable, then
the modulus of gradient coincides with $\|D f(x)\|_*$.  We shall work under the following Hypothesis $(\mathcal H)$:  for every $m,n \in \N^*$ and for every locally Lipschitz function $f : X^{m+n} \rightarrow \R$, for $\mu^{m+n}$-almost every point $(x,y) \in X^m \times X^n$ :
\begin{equation*}
\vert\nabla f\vert(x,y) = \vert\nabla_x f\vert(x,y) + \vert\nabla_y f\vert(x,y).
\end{equation*}

This assumption holds in various cases : when $(X,d)$ is an open metric subset of a Minkowski space $\left(\R^n,\Vert.\Vert\right)$ and when $\mu$ is absolutely continuous with respect to Lebesgue's measure, or for  Riemannian manifolds when the measure is absolutely continuous with respect to the volume form (as a consequence of Rademacher's theorem of almost everywhere differentiability of Lipschitz functions). On the contrary, this hypothesis often fails in discrete settings.

\section{Sharp isoperimetric inequalities}

We start by recalling a couple of important results about extremal half-spaces 
for the isoperimetric problem. The first one below is due to Bobkov and Houdr\'e \cite{bobkh97scbi} and deals with the real line. Before stating it, we need to introduce some notations. Let $\mathcal M$ be the set of Borel probability measures on $\R$ which are concentrated on a possibly unbounded interval $(a,b)$ and have a density $f$ which is positive and continuous on $(a,b)$. For $\mu\in \mathcal M$, the distribution function
$F_\mu(x):= \mu((-\infty,x])$ is one-to-one from $(a,b)$ to $(0,1)$ and one may define
   $$ J_\mu(t)=f\big(F_\mu^{-1}(t)\big),\; t\in (0,1).$$
We may as well consider $J_\mu$ as a function on $[0,1]$ by setting $J_\mu(0)=J_\mu(1)=0$.
The value of $J_\mu(t)$ represents the boundary measure of the half-line of measure $t$
starting at $-\infty$. 
 Let $\mathcal L\subset \mathcal M$ denote the set of (non-Dirac) log-concave probability measures on $\R$ (the density $f$ is of the form $e^{-c}$ for some convex function $c$).

\begin{prop}[\cite{bobkh97scbi}]\label{prop1d}
The map $\mu\mapsto J_\mu$ is one-to-one between the set $\mathcal M$ and the
set of positive continuous functions on $(0,1)$. It is also one-to-one between the subset 
$\mathcal L$ of log-concave probability measures and the set of positive concave functions
on $(0,1)$.
Moreover for $\mu\in \mathcal M$,  the  following properties are equivalent :
\begin{enumerate}[label=(\roman*)]
\item $I_{\mu} = J_{\mu}$  (meaning for any $p \in \left(0,1\right)$, the infimum in \eqref{a} is attained for the set  $\left(-\infty,F_\mu^{-1}(p)\right]$,
\item the measure $\mu$ is symmetric around its median, i.e. $J_{\mu}$ is symmetric around $\frac{1}{2}$, and for all $p,q > 0$ such that $p+q < 1$,
\begin{equation*}
J_{\mu}(p+q) \leq J_{\mu}(p) + J_{\mu}(q).
\end{equation*}
\end{enumerate}
\end{prop}

The next basic lemma allows to compare the various conditions on isoperimetric profiles that appear in the rest of the article. In particular, it shows that the above result encompasses a classical theorem of Borell, asserting that for even log-concave probability measures on $\R$,
half-lines are solutions to the isoperimetric problem.
\begin{lem}\label{lem:K}
Let $T\in (0,+\infty]$ and $K:[0,T)\to \R^+$ be a non-negative function. Consider the 
following properties that $K$ may verify :
\begin{enumerate}[label=(\roman*)]
 \item $K$ is concave,
 \item $t\mapsto K(t)/t$ is non-increasing,
 \item for all $x,y\in [0,T)$ with $a+b<T$, it holds $K(a+b)\le K(a)+ K(b)$.
\end{enumerate}
Then $(i)\Longrightarrow (ii)$ and $(ii)\Longrightarrow (iii)$.
\end{lem}
\begin{proof}
 If $K$ is concave then $t\mapsto (K(t)-K(0))/t$ is non-increasing. Since $t \mapsto K(0)/t$ is
 non-increasing as well, the first implication follows.
 Assuming $(ii)$ and without loss of generality $a\le b$,
 $$ K(a+b)\le (a+b) \frac{K(b)}{b}=a \frac{K(b)}{b}+K(b)\le K(a)+K(b).$$
\end{proof}

The next result provides sharp isoperimetric inequalities in high dimensions. It goes back to the dissertation thesis of S. Bobkov. See also \cite{bobk97ipue}.
\begin{theo}\label{th:ab}
Let $J:[0,1]\to \R^ +$ be a concave function, with $J(t)=J(1-t)$ for all $t\in[0,1]$.
Assume that for all $a,b\in [0,1]$,
 \begin{equation}\label{eq:I(ab)}
  J(ab)\le aJ(b)+bJ(a).
 \end{equation}
 Then for every space $(X,d,\mu)$ verifying Hypothesis ($\mathcal H$),
   $$ I_\mu \ge J \Longrightarrow I_{\mu^\infty}\ge J.$$
   Moreover there exists an even log-concave probability measure $\nu$ on $\R$ such that
   $I_\nu=I_{\nu^ \infty}=J$ and for every $n$, coordinate half-spaces are solutions of the isoperimetric
   problem for $\nu^ n$. 
\end{theo}

Condition~\eqref{eq:I(ab)} may be verified in a few instances as $J(t)=t(1-t)$. However, it is not so easy to deal 
with, in particular in conjunction with the symmetry assumption. For these reasons, stronger conditions of more local 
nature are useful. In \cite{bart04idii}, it is shown that \eqref{eq:I(ab)} is verified when $J$ is concave, twice differentiable
and $-1/J"$ is concave. Observe that condition \eqref{eq:I(ab)} amounts to the subadditivity 
of the function $u\mapsto e^uJ(e^{-u})$ on $\R^+$. Hence, using the second part of Lemma~\ref{lem:K}, we obtain that the  
  condition "$t\mapsto J(t)/(t\log(1/t))$ is non-decreasing" implies \eqref{eq:I(ab)} as well.
  By a tedious but straightforward calculation,  this 
  yields a neat variant of one of the main results of \cite{bart04idii}:
 \begin{cor}\label{coro:exinfdim}
  For $\beta\in [0,1]$, the function
  $K_{\beta}$ defined for $t\in [0,1]$ by
    $$ K_{\beta}(t):=  t(1-t)  \log^\beta \Big( \frac{3}{t(1-t)}\Big),$$
    satisfies that for every space $(X,d,\mu)$ verifying Hypothesis ($\mathcal H$) and  all $c\ge 0$,
$$I_\mu \ge c  K_{\beta} \Longrightarrow I_{\mu^\infty}\ge c  K_{\beta} .$$
\end{cor}

Let us point out that \eqref{eq:I(ab)} is not the best sufficient condition for the conclusion of 
the above theorem to hold. The optimal condition given by Bobkov's approach is the following :
for every Borel probability measure $N$ on $[0,1]$,
$$J\left(\int t \, dN(t)\right) \le \int J(t)\, dN(t)+\int_0^1 J\big(N([0,t])\big) \, dt.$$
Actually when $\mu\in \mathcal F$ is a probability measure on $\R$ and  $J=J_\mu=I_\mu$,
it is not hard to check, considering subgraphs, that the above condition is necessary and
sufficient for having $I_\mu=I_{\mu^\infty}$. However this  condition is hard to verify in practice,
and most of the work in Bobov's proof consists in showing that when $J$ is concave, it boils down to \eqref{eq:I(ab)}.

Next, we develop a different approach to dimension free isoperimetric inequalities.
We use classical methods to make a link between isoperimetric  inequalities, and some Beckner-type functional inequalities, which nicely tensorize.

\begin{lem}\label{lem:func}
Let $a\in(0,1]$ and $(X,d,\mu)$ be a metric probability space. Let $c>0$, then the following assertions are 
equivalent :
\begin{enumerate}[label=(\roman*)]
\item For all $p\in[0,1]$, $cI_{\mu}(p)\ge p-p^{\frac1a}$,
\item For every locally Lipschitz function $f:X\to [0,1]$, 
 $c\int |\nabla f| \, d\mu \ge \int f\, d\mu -\left( \int f^a \, d\mu \right)^{\frac1a}.$
\end{enumerate}
\end{lem}

\begin{proof}
Assuming $(i)$, we apply the co-area inequality to an arbitrary locally Lipschitz function $f$ (see e.g. \cite{bobkh97scbi}); next
we take advantage of the isoperimetric inequality for $\mu$:
\begin{eqnarray*}
 c\int |\nabla f| \, d\mu & \ge&  c\int_0^1 \mu^+(\{f\ge t\}) \, dt \\
   & \ge& \int_0^1  \Big(  \mu(\{f\ge t\})- \mu(\{f\ge t\})^{\frac1a} \Big) dt \\
   &= & \int f \, d\mu -  \int_0^1  \mu(\{f\ge t\})^{\frac1a} dt. 
\end{eqnarray*}
In order to conclude that the second assertion is valid, we apply the Minkowski inequality with exponent $1/a\ge 1$:
\begin{eqnarray*}
  \left(  \int_0^1  \mu(\{f\ge t\})^{\frac1a} dt \right)^a &=&  
      \left(  \int_0^1 \left( \int  {\bf 1}_{f(s)\ge t} \, d\mu(s)\right)^{\frac1a} dt \right)^a\\
      &\le & \int   \left(   \int_0^1   ({\bf 1}_{f(s)\ge t}) ^{\frac1a}  dt \right)^a d\mu(s) = \int f^a d\mu.
\end{eqnarray*}

\medskip
The fact that the second assertion implies the first one is rather standard : one applies the functional inequalities
to Lipschitz approximations of the characteristic function of an arbitrary Borel set $A \subset X$ (see Lemma 3.7 in \cite{bobkh97scbi}).
 This yields $c\mu^ +(A) \ge \mu(A)-\mu(A)^{\frac1a}$.
Applying the inequality to $1-f$ instead of $f$ and using $|\nabla f|=|\nabla (1-f)|$ and then taking approximations of $\mathbf{1}_A$
gives  $c\mu^ +(A) \ge 1- \mu(A)-(1-\mu(A))^{\frac1a}$ for all $A$, which is equivalent to $c\mu^ -(A) \ge \mu(A)-\mu(A)^{\frac1a}$
for all Borel sets $A$.
\end{proof}

\noindent The following extension of the classical subadditivity property of the variance is due to Lata\l a and Oleszkiewicz \cite{latao00bsp}. It allowed them to devise functional inequalities with the tensorization property. Actually, they focused
on Sobolev inequalities involving $L_2$-norms of gradients, with applications to concentration inequalities. Here we aim at
functional inequalities involving $L_1$-norms of gradients and provide information about isoperimetric inequalities.

\smallskip

\begin{lem}\label{lem:LO}
Let $\left(\Omega_1,\mu_1\right)$ and $\left(\Omega_2,\mu_2\right)$ be probability spaces and let $\left(\Omega,\mu\right)$ = $\left(\Omega_1 \times \Omega_2,\mu_1 \otimes \mu_2\right)$ be their product probability space. For any non-negative random variable $Z$ defined on $\left(\Omega,\mu\right)$ and having finite first moment and for any strictly convex function $\phi$  on $[0,+\infty)$ such that $\frac{1}{\phi''}$ is a concave function, the following inequality holds true :
\begin{equation*}
\mathbb{E}_{\mu}\phi(Z) - \phi\left(\mathbb{E}_{\mu}Z\right) \leq \mathbb{E}_{\mu}\Big(\big[\mathbb{E}_{\mu_1}\phi(Z) - \phi\left(\mathbb{E}_{\mu_1}Z\right)\big] + \big[\mathbb{E}_{\mu_2}\phi(Z) - \phi\left(\mathbb{E}_{\mu_2}Z\right)\big]\Big).
\end{equation*}
\end{lem}

\begin{theo}\label{th:tens-a-fixe}
Let $(X,d,\mu)$ be a metric probability space verifying hypothesis ($\mathcal H$).
Let $a\in[\frac12,1]$ and $c>0$. If for all $p\in (0,1)$, $I_\mu\ge c (p-p^{\frac1a})$, 
then  for all $p\in (0,1)$, $$I_{\mu^\infty}(p) \ge c (p-p^{\frac1a}).$$ 
\end{theo}

\begin{proof}
By Lemma \ref{lem:func}, we know that for every locally Lipschitz function $f:X\to [0,1]$
\begin{equation}\label{eq:LO}
\frac1c \int |\nabla f| \, d\mu \ge \int f\, d\mu -\left( \int f^a \, d\mu \right)^{\frac1a}.
\end{equation}

 We shall prove that this functional inequality tensorizes, meaning that for all $n$ the same property is verified by $\mu^n$. Applying Lemma \ref{lem:func} again 
will give the claimed dimension-free isoperimetric inequality.

Checking  the tensorization property is done along the same lines as in \cite{latao00bsp}. Assume that $(X_1,\nu_1,d_1)$ and $(X_2,\nu_2,d_2)$ satisfy \eqref{eq:LO}.
Since $a\in[\frac12,1]$, Lemma \ref{lem:LO} applies to $\Phi(t)=t^{\frac1a}$ and gives
\begin{eqnarray*}
&&\int f\, d\nu_1d\nu_2 -\left(\int  f^a d\nu_1d\nu_2\right)^{\frac1a} 
=
\int \Phi(f^a)\, d\nu_1d\nu_2 -\Phi\left(\int  f^a d\nu_1d\nu_2\right) \\
&\le & \int \left( \int \Phi(f^a)\, d\nu_1 -\Phi\left(\int  f^a d\nu_1  \right)\right) \,d\nu_2+
\int \left( \int \Phi(f^a)\, d\nu_2 -\Phi\left(\int  f^a d\nu_2  \right)\right) \,d\nu_1\\
&\le & \frac{1}{c} \int \big(|\nabla_1 f|+|\nabla_2 f| \big) \, d\nu_1d\nu_2,
\end{eqnarray*} 
 where $|\nabla_i f|$ is the norm of the gradient of $f$ taken with respect to the $i$-th
 variable. When $\nu_1=\mu^m$ and  $\nu_2=\mu^n$, we may apply Hypothesis ($\mathcal{H}$) to replace the function in the latter integral by the norm of the full gradient
 $|\nabla f|$. This allows to show by induction that for all $n$, $\mu^n$ verifies 
 the claimed functional inequality.
\end{proof}

This result readily generalizes :

\begin{theo}\label{th:becknergen}
Let $c:[\frac12,1]\to \R^+$, and consider for $p\in [0,1]$,
$$L(p):=\sup_{a\in[\frac12,1]} c(a) \max\big\{p-p^{\frac1a},1-p-(1-p)^{\frac1a}\big\}.$$
If $(X,d,\mu)$ satisfies ($\mathcal H$) and $I_\mu\ge L$ then 
  $$I_{\mu^\infty}\ge L.$$
  Moreover there exists an even probability measure $\nu$ on $\R$ such that
  $I_\nu=I_{\nu^\infty}=L$ and such that for all $n$, coordinate half-spaces are
  solutions to the isoperimetric problem for $\nu^n$.
\end{theo}

\begin{proof}
Observe that since, by definition, isoperimetric functions of probability measures
are symmetric with respect to $\frac12$,  the property for all $p\in [0,1]$,  $I_\mu\ge c (p-p^{\frac1a})$ is equivalent to $ I_\mu(p)\ge c M_a(p)$, for all $p$,
where 
$$M_a(p):=\max\big\{p-p^{\frac1a}, 1-p-(1-p)^{\frac1a}\big\}.$$
Hence the fact that $I_\mu\ge L$ implies $I_{\mu^\infty}\ge L$ is a direct consequence
of the previous theorem, applied for all values of $a$.

Next, it is not hard to check that for $a\in[\frac12,1]$, $M_a$ is subadditive, being a supremum of two concave functions defined on $\left[0,1\right]$. And, since the property "$J(x+y)\le J(x)+J(y)$ for all $x,y$" is stable under supremum, it follows that $L$ is also subadditive.

Hence, by Proposition \ref{prop1d}, there exists an even probability measure $\nu$ on $\R$ such that
$I_{\nu}=L$ and half-lines solve the isoperimetric problem for $\nu$.
As we just proved, $I_\nu\ge L$ ensures that $I_{\nu^\infty}\ge L$. Combining this with $L= I_\nu\ge I_{\mu^{\otimes \infty}}$ yields  $I_{\nu^\infty}= L$.
 The coordinate halfspace $\{x\in \R^n\,\Big| \,\; x_1\le t\}$ has same measure and boundary measure (for $\nu^n$), as the set $(-\infty, t]$ (for $\nu$).  It is then clear that it solves the isoperimetric problem.
  
\end{proof}
 Remark that for $a\in(\frac12,1)$, the function $M_a(p)=\max\big\{p-p^{\frac1a}, 1-p-(1-p)^{\frac1a}\big\}$  is not concave, hence the measure $\nu_a$ is not log-concave.
 Actually, $M_a$ does not even have its maximum at $\frac12$. Hence it cannot be 
 obtained as a supremum of concave functions which are in addition symmetric around $1/2$. 
Therefore it gives a genuinely new example of a measure for which coordinate
half-spaces solve the isoperimetric problem in any dimension (that could not be deduced
from Theorem \ref{th:ab}).

\section{Approximate inequalities}

 Let us start with some notations. Given two non-negative functions $f,g$ defined on a set $S\subset \R$ and $D\ge 1$, we write $f\approx_D g$ and say that $f$ and $g$ are equivalent up to a factor $D$ if there exists $a>0$ such that for all $x\in S$,
 $a\, g(x)\le f(x)\le D a\, g(x)$. We write $f\approx g$ when there exists $D$ such that 
 $f\approx_D g$.
 
\noindent We say that a non-negative function $f$ defined on a set  $S\subset \R$ is 
 essentially non-decreasing (with constant $D\ge 1$) when there exists a non-decreasing
 function $g$ on $S$ such that $f\approx_D g$. In the same way, we may
 define the notion of essentially non-increasing functions.
 
\noindent Also, a non-negative function $f$
defined on an interval is said to be essentially concave (or pseudoconcave) if it is equivalent
to a concave function. 

The next proposition provides workable formulations of the above definitions. The part about
essentially concave functions is due to Peetre \cite{peet68if2}.

\begin{lem}\label{lem:peetre}
 Let $f$ be a non-negative function defined on  $S\subset \R$.
 Then $f$  is essentially non-decreasing (resp. essentially non-increasing) with constant $D \geq 1$ if and only if for every $s \leq t $ in $S$,
\begin{equation*}
f(s)\leq Df(t) \quad \big(\text{resp.~} f(t) \geq Df(s)\big).
\end{equation*} 

When $f$ is defined on $(0,+\infty)$, the following assertions are equivalent :
\begin{enumerate}[label=(\roman*)]
\item $f$ is essentially concave with some constant $C_1$,
\item There exists $C_2\ge 1$ such that for all $ s,t \in \R^*_+$ , $\displaystyle f(s) \leq C_2 \max \left(1,\frac{s}{t}\right)f(t).$
\item There exists $C_3\ge 1$ such that on $\R_+^*$,  
 $f$ is essentially non-decreasing  and  $t \mapsto \frac{f(t)}{t}$ is essentially non-increasing, both with constant $C_3$.
\end{enumerate}
Moreover, the smallest possible constants verify  $C_1/2\le C_2=C_3 \leq C_1$.
\end{lem}

\begin{proof}
The argument for essentially non-decreasing functions is very simple and we skip it. Let us just point out that it involves the least non-decreasing function above $f$, which is given 
by $\acute f(t):=\sup\{f(x)\,\mid\,  x\in S\cap(-\infty,t]\}$.

Next let us focus on concavity issues. The equivalence of the last two statements is obvious.

 Assume $f$ is essentially concave on $\R^*_+$. Then there exists a concave function $h$ on $\R^*_+$ which is equivalent to $f$. And as $f$ is positive, $h$ is positive, therefore, being concave, $h$ is necessarily non-decreasing on $\R^*_+$. Moreover $t \mapsto \frac{h(t)}{t}$ is non-increasing on $\R^*_+$. So $f$ satisfies the third condition.

 Eventually, let us assume the second condition and show that $f$ is equivalent to a concave 
 function. The natural guess is the   least concave majorant of $f$,  which is explicitly given
 for $t>0$ by 
\begin{equation*}
\widehat{f}(t) := \sup\left\{\sum^{n}_{i=1}\lambda_if\left(t_i\right) \,\Big| \,  n \in \N^*,~ \lambda_i \geq 0, \, t_i>0,~ \sum^{n}_{i=1}\lambda_i=1 \text{~and~} \sum^{n}_{i=1}\lambda_it_i = t \right\}.
\end{equation*}
By definition $f \leq \widehat{f}$. Let $n \in \N^*$, $t \in \R^*_+$, $\left(\lambda_i\right)_{1 \leq i \leq n}$ and $\left(t_i\right)_{1 \leq i \leq n}$ such that, for all $i$, $\lambda_i \geq 0$, $\sum^{n}_{i=1}\lambda_i = 1$ and $\sum^{n}_{i=1}\lambda_it_i = t$. Using the hypothesis, we obtain
\begin{equation*}
\sum^{n}_{i=1}\lambda_if\left(t_i\right) \leq C_2\sum^{n}_{i=1}\lambda_i\max\left(1,\frac{t_i}{t}\right)f(t) \leq C_2\left(\sum^{n}_{i=1}\lambda_i + \sum^{n}_{i=1}\frac{\lambda_it_i}{t}\right)f(t) = 2C_2f(t)
\end{equation*}
Therefore $f\le \widehat{f}\le 2C_2 f$ and we have shown that $f$ is essentially concave.
\end{proof}

\medskip
We are now ready to state our main results :

\begin{theo}\label{th:dir}
Let $J$ be a non-negative function defined on $[0,1]$ with $J(0)=0$. Assume that it is symmetric around $\frac{1}{2}$ (i.e. for every $t \in \left[0,1\right], J(t) = J(1-t)$) and that the function $$ t\in(0,1)\mapsto \frac{J(t)}{t\log(1/t)}$$
is essentially non-decreasing with constant $D$.
 Then   for every metric probability space $(X,d,\mu)$ satisfying Hypothesis ($\mathcal H$):
 $$ I_\mu \ge J \Longrightarrow I_{\mu^\infty} \ge \frac{1}{c_D} J,$$ 
 with $ c_D=2(D/\log 2)^2\le 5D^2$.
 Moreover, there exists a symmetric log-concave measure $\nu$ on the real line such that, on $\left[0,1\right]$, $J \approx I_{\nu} \approx I_{\nu^{\otimes \infty}}$.

 If in addition $J$ is concave one can take $c_D=2D$ for $D>1$ and $c_1=1$.
\end{theo}

\begin{rem}
This result should be compared to a theorem of E. Milman in \cite{milm09rcfi}, where a similar
condition is given for dimension-free isoperimetric inequalities for the $\ell_2$ combination 
of distances on products (in other words for the Euclidean enlargement). His condition involves
an essential monotonicity property of $J/I_\gamma$ where $\gamma$ is the one-dimensional
standard Gaussian measure. On $(0,1/2]$ it is known that $I_\gamma(t)\approx t \sqrt{\log(1/t)}$.
\end{rem}

In order to formulate a converse statement, we introduce the following hypothesis : we say that $(X,d,\mu)$ enjoys the regularity property
($\mathcal R$) if  for all $t\in(0,1),$  $I_\mu(t)<+\infty$ and for all $n\in \N^*$, $t\in(0,1)$ and $\varepsilon >0$ there exists a Borel set $A\subset X^n$ with 
$\mu^n(A)=t$, $ (\mu^n)^+(A)\le I_{\mu^n}(t)+\varepsilon$ and
$$ (\mu^n)^+(A)=\lim_{h\to 0^+}\frac{\mu^n(A_h\setminus A)}{h},$$
where the products $X^n$ are equipped with the uniform distance. This hypothesis means that there are almost solutions of the isoperimetric problems for which the $\liminf$ in the definition of the Minkowski content is actually a real limit.
Thanks to Theorem 15 in \cite{bart02lsmi} it is not hard to check this property for  log-concave measures on the real line. We will give more comments on this hypothesis in Remark~\ref{rem:R}
below.

\begin{theo}\label{th:rev}
Let $(X,d,\mu)$ satisfy hypothesis ($\mathcal R$). Then the map
$$ t\in(0,1) \mapsto \frac{I_{\mu^ \infty}(t)}{t\log(1/t)}$$
is  continuous and essentially non-decreasing.
\end{theo}

Combining these two theorems, we can formulate our results as an equivalence :

\begin{cor}\label{cor:equiv}
Let $\left(X,d,\mu\right)$ denote a metric space equipped with a Borel probability measure $\mu$ and satisfying hypothesis ($\mathcal R$) and ($\mathcal H$). Then the following assertions are equivalent :

\begin{enumerate}[label=(\roman*)]
\item There exists a constant $C$ such that $\frac{I_{\mu}}{J_1}$ is essentially non-decreasing on $\left(0,1\right)$ with constant $C$,
\item There exists a constant $K \geq 1$ such that, on $\left[0,1\right]$, $\frac1K  I_{\mu} \leq  I_{\mu^{\infty}} \le   I_{\mu} $.
\end{enumerate}

\end{cor}

\medskip

We introduce two functions, both defined on $\left[0,1\right]$ by : $J_0(t) = t$ and $J_1(t) = t\log\frac{1}{t}$. The next lemma gives a different formulation of the main condition appearing in the previous theorems.

\begin{lem}\label{lem:equiv}
Let $K : \left[0,1\right] \rightarrow \R_+$ be a non-negative  function such that $K$ is symmetric with respect to $\frac12$ (i.e. for $t \in \left[0,1\right]$, $K(t) = K(1-t)$). Then the following assertions are equivalent :
\begin{enumerate}[label=(\roman*)]
\item There is a constant $C$ such that $\frac K{J_1}$ is essentially non-decreasing on $\left(0,1\right)$ with constant $C$.
\item There exists constants $C_0$ and $C_1$  such that $\frac K{J_0}$ is essentially non-increasing on $\left(0,\frac12\right]$ with
constant $C_0$ and 
$\frac K{J_1}$ is essentially non-decreasing on $\left(0,\frac12\right]$  with constant $C_1$.
\end{enumerate}
Moreover, the smallest possible constants verify $C \leq \frac{C_0C_1}{\log2}$ and  $C_0 \leq \frac{C}{\log2}$, $C_1\le C$.
\end{lem}

\begin{proof}
We use the concavity of the map $t \mapsto (1-t)\log\frac1{1-t}$, which yields, for every $t \in \left[0,\frac12\right]$, $t\log2 \leq (1-t)\log\frac1{1-t} \leq t$.
Assuming $(i)$, $\frac K{J_1}$ is essentially non-decreasing on $\left(0,\frac12\right]$ with constant $C$. For the second part of the assertion, let $0 < s \leq t \leq \frac12$. Then,
\begin{equation}\label{eq:pourlafin}
\frac{K(t)}{t} \leq \frac{K(1-t)}{(1-t)\log\frac1{1-t}} \leq C\frac{K(1-s)}{(1-s)\log\frac1{1-s}} \leq \frac{C}{\log2}\frac{K(s)}s.
\end{equation}

For the converse implication: assuming $(ii)$, we first check that $\frac K{J_1}$ is essentially non-decreasing on $\left[\frac12,1\right)$.
 Let $\frac12 \leq s \leq t < 1$, then
\begin{equation*}
\frac{K(s)}{s\log\frac1s} \leq \frac1{\log2} \frac{K(1-s)}{1-s} \leq \frac{C_0}{\log2}\frac{K(1-t)}{1-t} \leq \frac{C_0}{\log2}\frac{K(t)}{t\log\frac1t}.
\end{equation*}
To get the property on the whole interval $\left(0,1\right)$, it suffices to use $\frac12$ as an intermediate point.
\end{proof}

The next corollary describes the possible size of an infinite dimensional isoperimetric profile:

\begin{cor}\label{cor:infini}
Let $\left(X,d,\mu\right)$ denote a metric space equipped with a Borel probability measure $\mu$ and satisfying hypotheses ($\mathcal R$) and ($\mathcal H$). 

If $\displaystyle \inf_{t\in(0,\frac12]} \frac{I_\mu(t)}{t}=0$ then $I_{\mu^\infty}$ is identically  $0$, else there exist $\alpha,\beta>0$ such that for all $t\in[0,1]$,
\begin{equation*}
\alpha\min\left(t,1-t\right) \leq I_{\mu^\infty}(t) \leq \beta\min\left(t\log\frac1t,(1-t)\log\frac1{1-t}\right).
\end{equation*}
\end{cor}

\begin{rem}
The  function defined on $[0,1]$ by $t \mapsto \min\left(t,1-t\right)$ is the isoperimetric function of the double-sided exponential measure on $\R$, $e^{-|x|}dx/2$. Using   the notation and results of Corollary~\ref{coro:exinfdim}, we observe that  it  is equivalent to the function $K_0(t) =t(1-t)$. Moreover
there is a log-concave probability measure $\ell_0$ on the real line for which $K_0=I_{\ell_0}= I_{\ell_0^\infty}$ (actually, $\ell_0$ is the standard logistic measure $\ell$  with density $\frac{e^{-x}}{\left(1+e^{-x}\right)^2}$ with respect to Lebesgue's measure). Hence the lower bound is optimal up to the multiplicative factor.

The upper bound of $I_{\mu^\infty}$ given in the above corollary is due to Bobkov and Houdr\'e \cite{bobkh97scbi}. 
A similar remark applies to it:  the quantity in the upper estimate is equivalent to the function $K_1$ of Corollary \ref{coro:exinfdim}, 
which is also an infinite dimensional isoperimetric profile (of a measure which is reminiscent of Gumble 
laws, as its distribution function is of the order of  $ e^{-\beta e^{-y}}$ when $y \rightarrow -\infty$, for some  $\beta>0$).

The fact that the infinite dimensional isoperimetric profile is either trivial, or at least as big as the one of the exponential measure was already discovered, in  slightly different forms, by Talagrand \cite{tala91niic} and by Bobkov and Houdr\'e \cite{bobkh00wdfc}.
\end{rem}

\begin{proof}[Proof of Corollary~\ref{cor:infini}]
By Theorem \ref{th:rev}, there exists $C\ge 1$ such that for all $t\in(0,1/2]$,
 \begin{equation} \label{eq:I12}
  I_{\mu^\infty}(t) \le C t\log\Big(\frac{1}{t}\Big)\times \frac{2}{\log 2} \,  I_{\mu^\infty}\Big(\frac12\Big) .
\end{equation} 
 Applying Theorem \ref{th:rev} again, together with Lemma~\ref{lem:equiv}, we get that there
 exists $D\ge 1$ such that for all $t\in (0,1/2]$, 
   $$ D  \frac{I_{\mu^\infty}(t)}{t} \ge 2 I_{\mu^\infty}\Big(\frac12\Big).$$
Therefore, assuming $\inf_{t\in(0,\frac12]} \frac{I_\mu(t)}{t}=0$, and using that $I_\mu\ge I_{\mu^\infty}$, we can deduce that $ I_{\mu^\infty}\big(\frac12\big)=0$. Then \eqref{eq:I12} and the symmetry of isoperimetric functions yield $I_{\mu^\infty}=0$ pointwize.

Next assume that there exists $\kappa>0$ such that $I_\mu(t)\ge \kappa t$ for all $t\in(0,1/2]$. Then Theorem~\ref{th:dir} applies to $J(t):=\kappa \min(t,1-t)$  (Lemma~\ref{lem:equiv} gives a quick way to check the 
hypothesis) and gives $I_{\mu^\infty}\ge cJ$ for some $c>0$.
\end{proof}

\begin{rem}\label{rem:R}
Our results are stated for general metric spaces, but are devised for continuous settings
(e.g. for which the values taken by the measure cover all $[0,1]$). This is why additional hypotheses  appear in our statements.   One may find Hypothesis ($\mathcal H$) quite natural (it is related to a.e. differentiability of Lipschitz functions). On the other hand, Hypothesis ($\mathcal R$) is more demanding, as it seems to require approximation theorems by smooth sets. 

Let us point out  a possible variant of Theorem~\ref{th:rev} where all the hypotheses are incorporated
in the structure of the ambient space: assume that $X$ is a finite dimensional vector space of dimension $p$, that the distance $d$ is induced by a norm $N$ on $X$ and that $\mu$ has a positive $C^1$ density $h$ with respect to Lebesgue's measure, $\mu = h.\mathcal{L}^p$. We equip the product spaces $X^n$ with $d_{\infty}$, the $\ell_{\infty}$-combination of $N$, i.e. for $x,y \in X^n$, $d_{\infty}(x,y) = \underset{1\leq i\leq n}{\max}N\left(x_i,y_i\right)$. Then, instead of using the Minkowski content as a definition of the boundary measure, let us chose the notion of generalized
perimeter instead : if $A \subseteq X^n$ is measurable, then
\begin{equation*}
P_{\mu^{\otimes n},\infty} = \sup\left\{\int_A \sum^{n}_{i=1}\left\vert\nabla_i (\varphi h)\right\vert d\mathcal{L}^{np}\,\Big| \, \varphi \in \mathcal{C}^1_c\left(X^n\right) \mbox{ and } \underset{x\in X^n}{\sup}d_{\infty}(\varphi(x),0)\leq 1\right\},
\end{equation*}
where $\left\vert \nabla f \right\vert$ is the modulus of gradient of $f$.

\noindent Since the perimeter is defined as a supremum (recall that the Minkowski content is an inferior limit), the proof of Lemma~\ref{lem:2PI} below does not require any regularity assumption. Hence the proof of Theorem~\ref{th:rev} applies without any changes and does not  require ($\mathcal R$).
   The proof of Theorem~\ref{th:dir} also applies to this new setting, without assuming ($\mathcal H$), with the following main modification: instead of using functional inequalities for locally
   Lipschitz functions, we work in the class of functions of bounded variations. We refer the reader to the book of Ambrosio, Fusco and Pallara \cite{AFPBVF00} for an exhaustive study of this approach in the Euclidean case. This requires to use various results about these functions: co-area inequality (Theorem 3.40), approximation by smooth functions (Theorem 3.9), approximate differentiability (Theorem 3.83 and Proposition 3.92 among others.  
\end{rem}

\subsection{Proof of Theorem~\ref{th:dir}}

We start with a few preliminary statements. 

\begin{lem}\label{lem:tech}
Consider a function $K:[0,1]\to \mathbb R^+$ with $K(0)=0$. Assume that $K$ is symmetric
with respect to $\frac12$ and that  $\frac{K}{J_1}$ is essentially non-decreasing on $\left(0,1\right)$ with constant $D$. Then
\begin{enumerate}[label=(\roman*)]
\item $K$ is essentially non-decreasing on $[0,\frac12]$ with constant $\frac{2D}{e\log 2}$,
\item $K$ is essentially concave. More precisely there exists a concave function $I:[0,1]\to \mathbb R^+$, which is symmetric with respect to $\frac12$, and is equivalent to $K$ up to a factor $2D/\log 2$.
\end{enumerate}
\end{lem}

\begin{proof}
 Observe that the  function $J_1(t)=t\log(1/t)$ is increasing on $(0,1/e]$ and decreasing on $[1/e,1)$. Its maximum is therefore $J(1/e)=1/e$.

Assume that $0 \leq s \leq t \leq \frac{1}{2}$.  Then, by hypothesis  $K(s) \leq D\frac{J_1(s)}{J_1(t)}K(t)$.  If  $ t \leq \frac{1}{e}$, we can conclude that  $K(s) \leq D K(t)$. If $t\in (1/e,1/2]$,
we argue differently
$$K(s) \leq D\frac{J_1(s)}{J_1(t)}K(t) \leq   D\frac{J_1(1/e)}{J_1(1/2)}K(t) =\frac{2D}{e\log 2}K(t).$$
This concludes the proof of $(i)$.

Next, let us prove $(ii)$. Consider  the map $\widetilde{K}$ defined on $\R_+$ by :
$$
\widetilde{K}(t) = \left\{
              \begin{array}{ll}
               K(t) &\text{~if~} t \in \left[0,\frac{1}{2}\right) \\
               K\left(\frac{1}{2}\right) &\text{~if~} t \ge  \frac{1}{2}
              \end{array}
              \right.
$$
 Combining $(i)$ and the second part of  $(ii)$ in Lemma \ref{lem:equiv},  one readily checks  that $\widetilde{K}$ satisfies the hypothesis of  Assertion $(iii)$ in  Lemma \ref{lem:peetre} with constant $\frac{D}{\log 2}$ ($\geq \frac{2D}{e\log2}$). Hence there exists a concave function $H$ which is  equivalent to $\widetilde{K}$ on $\left(0,+\infty\right)$, up to a factor $2D/\log 2$.
 Define $I$ to be the restriction of $H$ to $\left(0,\frac{1}{2}\right]$, extended at $0$ by $I(0)=0$ and to $\left[0,1\right]$ by symmetry with respect to $\frac{1}{2}$. Since $H$ is concave and non-negative on $(0,+\infty)$, it is also non-decreasing. Therefore, the function $I$ is concave as well. As
 $\widetilde K\approx H$ on $(0,+\infty)$, we obtain by restriction that  $K\approx I$ on $(0,1/2]$, up to the same constant. Since $K(0)=I(0)=0$, and both $I$ and $K$ are symmetric with respect to $1/2$, we can conclude that $I \approx K$ on $[0,1]$, up to a factor $2D/\log 2$.
\end{proof}

The following result shows how we exploit the essentially monotonicity properties of $J/J_0$ and   $J/J_1$
where $J_0(t)=t$ and $J_1(t)=t\log(1/t)$.
\begin{prop}\label{prop:Jsup}
Let $J:(0,1)\to \mathbb R^+$ such that for all $t$, $J(t)=J(1-t)$. Assume that on $(0,1/2]$, $J/J_0$  is essentially non-increasing with constant $D_0$ and $J/J_1$ is essentially non-decreasing with constant $D_1$.
Then there exists a function $c:[1/2, 1)\to \mathbb R^+$ such that for all $t\in (0,1)$,
   $$J(t)\ge \sup_{a\in  [1/2, 1)} c(a) \big(t-t^{\frac1a}\big),$$
   and for all $t\in (0,\frac12]$,
   $$  J(t)\le   2D_0\max(D_0,D_1) \sup_{a\in  [1/2, 1)} c(a) \big(t-t^{\frac1a}\big).$$
\end{prop}

 The proof of this proposition relies on the following statement, which is related to  \cite[Lemma 19]{bartcr06iibe}. 

\begin{lem}\label{lem:bcr}
Let $\Phi:\big(0,\frac{1}{\log 2}\big]\to \mathbb R^+$. If $\Phi$ is  essentially non-increasing with constant $C_0$, then for all $y\in(0,\frac12]$,
$$ \sup_{\alpha\in(0,1]}  \Phi(\alpha) \big(1-y^\alpha\big)\ge \frac{1}{2C_0}\,  \Phi\left( \frac{1}{\log(1/y)}\right).$$
If, in addition, $s\mapsto s\Phi(s)$ is essentially non-decreasing with constant $C_1$, then for all $y\in (0,1/2]$,
$$ \sup_{\alpha\in(0,1]}  \Phi(\alpha) \big(1-y^\alpha\big )\le \max(C_0,C_1) \,\Phi\left( \frac{1}{\log(1/y)}\right).$$

\end{lem}

\begin{proof} 
In order to bound the supremum from below, we just select an appropriate value for $\alpha$:
if $y\in (0,1/e]$, choosing $\alpha=1/\log(1/y)\in(0,1]$, we get  
$$ \sup_{\alpha\in(0,1]}  \Phi(\alpha) \big(1-y^\alpha\big)\ge (1-e^{-1}) \Phi\left( \frac{1}{\log(1/y)}\right).$$
If $y\in[1/e,1/2]$, we choose $\alpha=1$ and get 
$$ \sup_{\alpha\in(0,1]}  \Phi(\alpha) \big(1-y^\alpha\big)\ge (1-y) \Phi(1)\ge \frac12 \Phi(1).$$
However,  $y\ge 1/e$ ensures $1/\log(1/y)\ge 1$. Since $\Phi$ is essentially non-increasing, 
$C_0 \Phi(1)\ge \Phi(1/\log(1/y))$. Therefore we have proved the first claim, with a constant
$\min( 1-e^{-1}, 1/(2C_0))= 1/(2C_0)$.

To prove the converse inequality, we change variables as follows: setting $c=\alpha \log(1/y)$, 
$$  \sup_{\alpha\in(0,1]}  \Phi(\alpha) \big(1-y^\alpha\big) =\sup_{c\in(0,\log(1/y)]} (1-e^{-c}) 
\Phi\left( \frac{c}{\log(1/y)}\right).$$
For $c\ge 1$, we know that $ \Phi\left( \frac{c}{\log(1/y)}\right)\le C_0  \Phi\left( \frac{1}{\log(1/y)}\right)$ and we bound $1-e^{-c}$ from above by 1.
For $c\in (0,1]$, we take advantage of the hypothesis on $x\mapsto x\Phi(x)$, in the form
$ c\Phi(cx)\le C_1 \Phi(x)$ for $x>0$:
$$ (1-e^{-c}) 
\Phi\left( \frac{c}{\log(1/y)}\right) \le C_1 \frac{1-e^{-c}}{c}  \Phi\left( \frac{1}{\log(1/y)}\right)
\le C_1  \Phi\left( \frac{1}{\log(1/y)}\right).$$
These two estimates readily give the claim.

\end{proof}

\begin{proof}[Proof of Proposition \ref{prop:Jsup}]
For $\alpha\in(0,1/\log 2]$, we define
$$ \Phi(\alpha):= \frac{J\left( e^{-1/\alpha}\right)}{e^{-1/\alpha}}=\frac{J}{J_0}\left( e^{-1/\alpha}\right).$$
Since $e^{-1/\alpha}\in(0,1/2]$, our hypothesis ensures that $\Phi$ 
is essentially non-increasing with constant $D_0$. Notice that
$$ \alpha \Phi(\alpha):= \frac{J\left( e^{-1/\alpha}\right)}{e^{-1/\alpha}\log\left(\frac{1}{e^{-1/\alpha}} \right)}=\frac{J}{J_1}\left( e^{-1/\alpha}\right)$$
Hence by hypothesis, it is essentially non-decreasing with constant $D_1$. Therefore we may apply
the previous lemma to $\Phi$. Since by definition, $\Phi(1/\log(1/y))=J(y)/y$, it gives that
for all $y\in(0,1/2]$, 
$$  \frac{J(y)}{y} \ge  \frac{1}{\max(D_0,D_1)}\sup_{\alpha\in (0,1]} \Phi(\alpha)\big(1-y^\alpha\big),$$
and for all $y\in(0,1/2]$,
$$  \frac{J(y)}{y} \le  2D_0 \sup_{\alpha\in (0,1]} \Phi(\alpha)\big(1-y^\alpha\big).$$
Multiplying these inequalities by $y$, and setting $a:=1/(1+\alpha)$, the former estimate 
gives for $y\in(0,1/2]$
\begin{equation}\label{eq:lowerJ}
  J(y)\ge  \frac{1}{\max(D_0,D_1)}\sup_{\alpha\in (0,1]} \Phi(\alpha)\big(y-y^{1+\alpha}\big)
= \frac{1}{\max(D_0,D_1)}\sup_{a\in [1/2,1)} \Phi\left(\frac{1}{a}-1\right)\big(y-y^{\frac1a}\big).
\end{equation}
Hence we have prove the claimed lower bound on $J$ with $c(a)=\Phi(a^{-1}-1)/\max(D_0,D_1)$. 
We proceed in the same way with the upper bound on $J$. 
The ratio of the upper bound to the lower bound is $2 D_0 \max(D_0,D_1)$.

It remains to extend the lower bound \eqref{eq:lowerJ} to values $y\in(1/2,1)$. To do this 
we use the symmetry of $J$ and the fact that for all $a\in[1/2,1]$ and all $s\in[1/2,1)$, 
$1-s-(1-s)^{\frac1a}\ge s-s^{\frac1a}$ (this follows from the comparison of second derivatives,
observing that equality holds at   $1/2$ and 1): for $y\in(1/2,1)$,
$$ J(y)=J(1-y)\ge c(a) \big(1-y-(1-y)^{\frac1a}\big) \ge c(a) \big(y-y^{\frac1a}\big).$$
\end{proof}

\bigskip

\begin{proof}[Proof of Theorem \ref{th:dir}]
Let us denote by $D_0\ge 1$ the smallest constant such that $J/J_0$ is essentially non-increasing on $(0,1/2]$ with constant $D_0$.
Similarly let $D_1\ge 1$ be the smallest constant such that $J/J_1$ is essentially non-decreasing on $(0,1/2]$ with constant $D_1$.
 
 First, we apply Proposition \ref{prop:Jsup}. With the notation of the proposition,
it follows that for all $a\in [1/2,1)$ and all $t\in[0,1]$,
$$ I_\mu(t)\ge J(t)\ge c(a) \big( t-t^{\frac1a}\big).$$ 
Note that for $t=0$ or $t=1$ all quantities vanish. Next, Theorem~\ref{th:tens-a-fixe} tell us 
that $ I_{\mu^\infty}(t)\ge c(a) \big( t-t^{\frac1a}\big).$ This is true for all $a$ and all $t$, hence
applying the second part of Proposition \ref{prop:Jsup}, we deduce that for all $t\in[0,1/2]$,
$$  I_{\mu^\infty}(t)\ge  \sup_{a\in [1/2,1)}c(a) \big( t-t^{\frac1a}\big) \ge \frac{1}{2D_0\max(D_0,D_1)} J(t).$$
Since both $I_{\mu^\infty}$ and $J$ are symmetric with respect to $1/2$, we can conclude
that for all $t\in[0,1]$, it holds $I_{\mu^\infty}(t)\ge J(t)/(2D_0\max(D_0,D_1))$.

In the general case, we know by Lemma~\ref{lem:equiv} that $D_0\le D/\log 2$ and $D_1\le D$. Therefore we get that 
$ I_{\mu^\infty} \ge J/c_D$ with $c_D= 2D^2/(\log 2)^2$.

 In the particular case where $J$ is concave, we know that  $J(t)/t$ is non-increasing, so that $D_0=1$ and $ I_{\mu^\infty}\ge J/(2D_1)\ge J/(2D).$ The paragraph after Theorem \ref{th:ab}
 explains that when $J$ is concave, symmetric and such that $J/J_1$ is non-decreasing, the inequality $I_\mu\ge J$ implies $I_{\mu^\infty} \ge J$. Hence in this case the conclusion of  Theorem  \ref{th:dir} is valid with $c_1=1$.
\end{proof}

\bigskip

\subsection{Proof of Theorem~\ref{th:rev}}

First, we recall a classical property of infinite dimensional profiles, which comes from 
testing isoperimetric inequalities on product sets. It was put forward by Bobkov in \cite{bobk97ipue}.

\begin{lem}\label{lem:2PI}
Let $\left(X,d,\mu\right)$ be a metric space equipped with a Borel probability measure $\mu$ and satisfying the regularity property ($\mathcal R$). Then the infinite-dimensional isoperimetric profile of $\left(X,d,\mu\right)$ satisfies, for every $a,b \in \left[0,1\right]$ :
\begin{equation}\label{eq:I(ab)2}
I_{\mu^{\infty}}(ab) \leq aI_{\mu^{ \infty}}(b) + bI_{\mu^{ \infty}}(a).
\end{equation}
\end{lem}

\begin{proof} The inequality is obvious if $a$ or $b$ is equal to 1, or to 0 since $I_{\mu^{ \infty}}(0)=0$. Let $a,b \in (0,1)$ and $\varepsilon>0$.
Let $m,n \in \N^*$. Let $A\subset X^m$ be any set with $\mu^m(A)=a$. Let $B\subset X^n$
be any set with $\mu^n(B)=b$, $(\mu^n)^+(B)\le I_{\mu^n}(b)+\varepsilon$ and 
$$(\mu^n)^+(B)=\lim_{h\to 0}\frac{\mu^n\left(B_h\setminus B\right)}{h},$$ which is 
possible thanks to Hypothesis ($\mathcal R$).

Then consider  $A~\times~B~\subseteq~\R^{m+n}$. Obviously, $\mu^{ m+n}\left(A \times B\right) = ab$. The uniform enlargement of a product set is still a product:  $ (A\times B)_h=A_h\times B_h$ for any  $h > 0$. Therefore
\begin{eqnarray*}
 \frac1h \mu^{m+n}\left( (A\times B)_h\setminus (A\times B)\right) &=& \frac{\mu^m(A_h)\mu^n(B_h)-\mu(A)\mu(B)}{h}\\
 &=&  \frac{\mu^m(A_h\setminus A)}{h}\, \mu^n(B_h)+ \mu^m(A) \frac{\mu^n(B_h\setminus B)}{h}
\end{eqnarray*}
Since by hypothesis $\lim_{h\to 0} (\mu^n(B_h)-\mu^n(B))/h<+\infty$, we know that  $\lim_{h\to 0} \mu^n(B_h)=\mu^n(B)$ (note the convergence holds by monotonicity). Taking upper limits in $h\to 0$, and observing that two of the three terms have limits, we deduce from   the latter inequality that
\begin{equation}\label{eq:bordprod}
 (\mu^{m+n})^+(A\times B)\le (\mu^m)^+(A)\, \mu^n(B)+\mu^m(A)\,(\mu^n)^+(B).
 \end{equation}
Since $I_{\mu^{m+n}}(ab)\le (\mu^{m+n})^+(A\times B)$, we obtain after optimizing on sets
$A$ of measure $a$ and using the hypothesis on the boundary measure of $B$:
$$
 I_{\mu^{m+n}}(ab) \leq I_{\mu^m}(a) \, b + a \big(I_{\mu^n}(b)+\varepsilon\big).$$
 Letting $\varepsilon$ tend to $0$, and $m,n$ tend to $+\infty$ gives the claim \eqref{eq:I(ab)2}. 
\end{proof}

The symmetry property ($I_{\mu^{\infty}}(t)=I_{\mu^{\infty}}(1-t)$ for all $t\in0,1]$) and 
the two-points inequality \eqref{eq:I(ab)2} are enough to deduce Theorem~\ref{th:rev}, as the
next statement shows:

\smallskip

\begin{prop}\label{prop:I}
Let $I:[0,1]\to[0,+\infty]$ be an application satisfying that for all $a,b\in[0,1]$
\begin{equation} \label{eq:hyp} 
I(a)=I(1-a) \quad \mathrm{and}\quad I(ab)\le aI(b)+bI(a),
\end{equation}
with the convention that $+\infty\times 0=0$.
If there exists $x_0\in [0,1]$ such that $\limsup_{x\to x_0} I(x)<+\infty$ then $I$ is continuous and
$t\mapsto I(t)/(t\log(1/t))$ is essentially non-decreasing on $(0,1)$.
\end{prop}
The condition of local boundedness around some point cannot be removed as shown by the following example:
$I(t)=0$ if $t\in \mathbb Q$ and $I(t)=+\infty$ otherwise.

\medskip
The proof of the proposition uses the next two easy lemmas.

\begin{lem}\label{lem:set}
 Let $S\subset (0,1)$ be a set with the following stability property:
 $$ \big(x\in S \quad \mathrm{and} \quad y\in S\big) \Longrightarrow \big(1-x \in S \quad\mathrm{and}\quad  xy\in S\big).$$
 If $S$ is not empty then it is dense in $(0,1)$. Moreover, if $S$ has non-empty  interior
 then $S=(0,1)$.
\end{lem}
In other words, if $S$ is neither $\emptyset$ nor $(0,1)$ then $S$ and $(0,1)\setminus S$ are dense in $(0,1)$.
This is the case for instance of $S=\mathbb Q\cap (0,1)$.

\begin{proof}
Let $t\in(0,1)$ be an element of $S$, then for all $n\in \mathbb N^*$, $x_n:=1-t^n$ belong to $S$ and the sequence $(x_n)$ tends to 1.
Given $0<a<b<1$, let us show that there is a point of $S$ between $a$ and $b$. Choose $k$ large enough such that $x_k> \max(b, a/b)$.
Then for all $n\ge 1$, $(x_k)^n\in S$. Obviously $x_k\ge b$ and $\lim_n (x_k)^n=0$. Let $n_0$ be maximal with $(x_k)^{n_0}\ge b$.
Then $$ b>x_k^{n_0+1} = x_k^{n_0} x_k > b \times \frac{a}{b}=a.$$
Hence $x_k^{n_0+1}\in S\cap (a,b)$. This completes the proof of the density of $S$.

Assume now that $(a,b)\subset S$ for some $0<a<b<1$. Consider an arbitrary $x\in (0,1)$ and let us show that $x\in S$.
If $x\in (a,b)$, we have nothing to prove. If $x\in (0,a]$, we use the fact that $S$ being non-empty is dense: there exists
$y\in S\cap (x/b, x/a)$. Since $S$ contains $y$ and $(a,b)$, the stability by product ensures that $S$ also contains $(ya,yb)$.
Hence $x\in(ya,yb)\subset S$.
Eventually, if $x\in [b,1)$, we consider $1-x\in (0,1-b]$. By the symmetry assumption $(1-b,1-a)\subset S$, so the latter argument
yields $1-x\in S$, and using symmetry again $x\in S$.
\end{proof}

\noindent The next lemma is a classical result about subadditive functions on $\R_+$ (see e.g. \cite{kuc09itfei}) :

\begin{lem}\label{lem:subadd}
Let $K : \R_+ \rightarrow \R$ be a subadditive function with $\lim_0 K=0$. Then
\begin{equation*}
\lim_{h \to 0^+}\frac{K(h)}{h} = \sup_{t > 0}\frac{K(t)}{t} \cdot
\end{equation*}
\end{lem}

\begin{proof}
Denote $S =\underset{t > 0}{\sup} \frac{K(t)}{t}$. Given any $u < S$, there exists $x_0 \in \R^*_+$ such that $K\left(x_0\right) > ux_0$. For any $h \in \left(0,x_0\right)$,
write  $x_0 = nh + \delta$ with  $n = \lfloor\frac{x_0}{h}\rfloor$ and $\delta\in[0,h)$.
By subadditivity of $K$,
\begin{equation*}
u x_0<K(x_0)\le n K(h)+K(\delta)= (x_0-\delta)  \frac{K(h)}{h} + K(\delta) .
\end{equation*}
Next, we  let  $h$ tend to $0^+$. In this case  $\delta \rightarrow 0^+$ and $K(\delta) \rightarrow 0$,
hence   (for any $u<S$)
\begin{equation*}
u \leq \liminf_{h \rightarrow 0^+} \frac{K(h)}{h} \cdot 
\end{equation*}
Therefore $S\le \liminf_{h \rightarrow 0^+}\frac{K(h)}{h}$.
On the other hand,  $\limsup_{h \rightarrow 0^+}\frac{K(h)}{h} \leq S$ holds by definition.
\end{proof}

\medskip

\begin{proof}[Proof of Proposition \ref{prop:I}]
There is nothing to prove if $I$ is identically 0, so we assume that $I$ does not vanish everywhere.
Observe that the two-points inequality in \eqref{eq:hyp}, applied for $a=b=0$, yields $I(0)=0$. 

\smallskip

Consider the subset $S_1$ of $(0,1)$ of points $x$ such that the function $I$ is bounded on a 
neighbourhood of $x$.  Our hypothesis $\limsup_{x\to x_0} I(x)<+\infty$ ensures that $S_1$
has non-empty interior. Thanks to \eqref{eq:hyp}, one readily checks that $S_1$ is stable by product and by symmetry with respect to $1/2$. Hence Lemma~\ref{lem:set} applies to $S_1$
and shows that $S_1=(0,1)$. This means that $I$ is locally bounded at every point of $(0,1)$.
By compactness, we deduce that $I$ is bounded on any segment $[a,b]\subset (0,1)$.

\smallskip

The next step of the proof is an argument of Bobkov and Houdr\'e, that we include for completeness. The two-points inequality implies by induction that for all $a\in [0,1]$ and 
any integer $k\ge 1$, $I(a^k)\le ka^{k-1} I(a)$. Let $t\in (0,1/e]$. Choosing $k=\lfloor \log(1/t)\rfloor\ge 1$ and $a=t^{1/k}$ in the latter  inequality leads to 
$$ I(t)\le k t^{1-\frac1k} I\big(t^{\frac1k}\big) = t \big\lfloor \log(1/t)\big\rfloor \frac{I\big(t^{1/k}\big)}{t^{1/k}}\cdot$$
Using that for $x\ge 1$, $x/\lfloor x \rfloor\in [1,2]$, we obtain that $t^{1/k}=\exp(-\log(1/t)/\lfloor \log(1/t)\rfloor)\in [e^{-2},e^{-1}]$. Hence for $t\in (0,e^{-1}]$,  $I(t)\le C t \log (1/t)$ 
where $C=e^2 \sup\{ I(s);\, s\in [e^{-2}, e^{-1}]\}$ is finite (by the previous point). In particular, this
estimates implies that $I(t)$ tends to 0 when $t\neq 0$ tends to $0$. Since $I(0)=0$, the function
$I$ is continuous at 0. By symmetry it is also continuous at $1$, with  $I(1)=0$.

\smallskip

 Consider  the map $K: \mathbb R^+\to \mathbb R^+$ defined by
$
K(x)=e^xI\left(e^{-x}\right).$  Then for all $x,y \in \R_+$, by \eqref{eq:hyp}
$$K(x+y) = \frac{I\left(e^{-x}e^{-y}\right)}{e^{-x}e^{-y}} \leq \frac{e^{-x}I\left(e^{-y}\right)+e^{-x}I\left(e^{-y}\right)}{e^{-x}e^{-y}} =K(y)+K(x),$$ which means that $K$ is subadditive.
Moreover, since $I$ is continuous at 1, $K$ is continuous at 0, with $K(0)=I(1)=0$.

For all $a > \epsilon > 0$, we have by subadditivity $K(a+\epsilon) \leq K(a) + K(\epsilon)$ and
 $K(a)\le  K(a-\epsilon) + K(\epsilon)$. Letting $\epsilon$ tend to zero, we obtain for all $a>0$
\begin{equation*}
\limsup_{x \rightarrow a^+} K(x) \leq K(a) \leq \liminf_{x \rightarrow a^-} K(x).
\end{equation*}
In words, on $(0,+\infty)$, the function $K$ is right-upper-semicontinuous and left-lower-semicontinuous. Since for all $t\in (0,1)$, $I(t)=t\, K(\log(1/t))$, if follows that on $(0,1)$
the function $I$ is left-upper-semicontinuous and right-lower-semicontinuous. Note that 
"left" and "right" were exchanged, since $t\mapsto \log(1/t)$ is continuous decreasing.
The symmetry assumption $I(t)=I(1-t)$ allows to exchange once more: so $I$ is also 
 right-upper-semicontinuous and left-lower-semicontinuous on $(0,1)$. Thus $I$ 
 is continuous on $(0,1)$, and actually on $[0,1]$. Indeed, the continuity at the endpoints has already
 been established.

\smallskip 

Next, let us draw another consequence of the above properties of $K$. Lemma~\ref{lem:subadd}
directly applies and gives that 
$$ \lim_{h\to 0^+} \frac{K(h)}{h}=\sup_{x>0} \frac{K(x)}{x}=\sup_{x>0} \frac{e^xI(e^{-x})}{x}.$$
Since we assume that $I$ is not identically 0, the above limit, denoted by $L$, belongs to  $(0,+\infty]$. We now translate this convergence in terms of $I$: using symmetry, for $t\in(0,1)$,
$$\frac{I(t)}{t}=\frac{I(1-t)}{t}=\frac{1-t}{t}K\left(\log\left(\frac{1}{1-t}\right) \right)
=\frac{(1-t)\log\left(\frac{1}{1-t}\right)}{t} \times   \frac{K\left(\log\left(\frac{1}{1-t}\right) \right)}{\log\left(\frac{1}{1-t}\right)} .$$
  When $t>0$ tends to $0$, the first ratio tends to 1, and the second to $L=\lim_{h\to 0^+} K(h)/h$. Therefore we can deduce that $\lim_{t\to 0^+} I(t)/t=L\in (0,+\infty]$.

\smallskip
  In order to turn this limit into a lower bound on $I(t)/t$ for $t\in (0,1/2]$, we need to check
 that $I$ does not vanish in $(0,1)$. To do this, let us consider the set $S_0=\{x\in (0,1); \; I(x)=0\}$. By \eqref{eq:hyp}, it is stable by product and symmetry around $1/2$. If it were non-empty,
 the first part of Lemma~\ref{lem:set}
would imply that $S_0$ is dense in $(0,1)$. By continuity of $I$, we would conclude that $I$
is identically 0. Since, we assumed that $I$ does not vanish everywhere, it follows that $S_0=\emptyset$.  As a conclusion, the function $I$ vanishes only at $0$ and $1$.
 
  On $(0,1/2]$ the map $t\mapsto I(t)/t$ is continuous, with  positive values. Moreover
  it has a positive (maybe infinite) limit at $0^+$. As a consequence, there exists $c>0$ such 
  that $I(t)\ge c t$ for all $t\le 1/2$.
  
\smallskip Let us deduce  that $I$ is essentially non-decreasing on $[0,1/2]$. Let $0\le s<t\le \frac12$. Using the two-points inequality \eqref{eq:hyp}
$$I(s)=I\left(t \times \frac{s}{t}\right)\le \frac{s}{t} I(t)+ t I\left(\frac{s}{t}\right)\le I(t)+ t \max I
\le  \left(1+\frac{\max I}{c}\right) \, I(t),$$
where we have used that $I$ is continuous on $[0,1]$ and $I(t)\ge c t$.

Eventually, let us prove that $\frac{I}{J_1}$ is essentially non-decreasing on $\left(0,1\right)$.
Let $0 < s <t < 1$. Then one can write  $s = t^{k+\alpha}$ with $k=\left\lfloor\frac{\log\frac1s}{\log\frac1t}\right\rfloor \in \mathbb N^*$ and $\alpha \in \left[0,1\right)$.
By the two-points inequality for $I$:
\begin{equation}\label{eq:Ist}
\frac{I(s)}{s} = \frac{I\left(t^{k+\alpha}\right)}{t^{k+\alpha}}  \leq k\frac{I(t)}t + \frac{I\left(t^{\alpha}\right)}{t^{\alpha}}.
\end{equation}

Assume first that $t\ge \frac12$. Then $t^{\alpha} \geq t \geq \frac12$. We have shown that 
$I$ is essentially non-decreasing on $[0,\frac12]$ (with a constant denoted by $D$).
By symmetry, it follows that $I$ is essentially non-increasing on $[\frac12,1]$ with constant $D$.
Hence
$$  \frac{I\left(t^{\alpha}\right)}{t^{\alpha}} \le D  \frac{I\left(t\right)}{t^{\alpha}} \le D
 \frac{I\left(t\right)}{t}.$$
 Combining this estimate with \eqref{eq:Ist} gives
 $$ \frac{I(s)}{s} \le ( k+D)\frac{I(t)}t \le (1+D) k\frac{I(t)}t \le (1+D) \frac{\log\frac1s}{\log\frac1t}\frac{I(t)}t,$$
 that is $\frac{I(s)}{J_1(s)}\le (1+D) \frac{I(t)}{J_1(t)}$. In particular, we have shown that 
 $I/J_1$ is essentially non-decreasing on $[1/2, 1]$. Using the symmetry of $I$, this implies
 that on $[0,1/2]$ the function $I(t)/J_0(t)=I(t)/t$ is essentially non-increasing. This is actually
 explained in the first part of the proof of Lemma~\ref{lem:equiv}, see Equation \eqref{eq:pourlafin}. We have already shown that $I$ is essentially non-decreasing on $(0,1/2]$.
 Thus by symmetry, $I$ is essentially non-increasing on $[1/2,1]$, and so is the map $t\mapsto I(t)/t=I(t)/J_0(t)$.
Therefore,  $I/J_0$ is essentially non-increasing on the whole interval $(0,1]$. Let us denote by $D_0$
the corresponding constant.

The latter fact allows to conclude: Let $0 < s <t < 1$. Since  $t^\alpha \ge t$, we know that
$ \frac{I(t^\alpha)}{t^\alpha} \le D_0  \frac{I(t)}t$. Combining this estimate with \eqref{eq:Ist} 
gives, 
$$
 \frac{I(s)}{s}  \leq k\frac{I(t)}t + \frac{I\left(t^{\alpha}\right)}{t^{\alpha}} \le (k+D_0) \frac{I(t)}t 
  \le (1+D_0) k\frac{I(t)}t  \le  (1+D_0) \frac{\log\frac1s}{\log\frac1t} \frac{I(t)}t  .$$ The proof is now complete.

\end{proof}

\section{An application to geometric influences}

This section is devoted to an application of Theorem \ref{th:dir} to geometric influences. The notion of influence of a variable on a boolean function plays an important role in discrete harmonic analysis, with applications to various fields (see e.g. the survey article \cite{kalsaf06} on threshold phenomena). Let us recall the definition: for a function $f:\{0,1\}^n\to \{0,1\}$, which can be viewed as a subset $A=\{x;\; f(x)=1\}$ of $\{0,1\}^n$, the influence of the $i$-th variable  with respect to a probability measure $\nu$ on the discrete cube $\left\{0,1\right\}^n$ is 
\begin{equation*}
I_i(f)=I_i(A) := \mathds{P}_{x\sim \nu}\big(f(x)\neq f(\tau_i(x))\big) =\mathds{P}_{x\sim \nu}\big(x\in A \mbox{ xor }\tau_i(x) \in A\big),
\end{equation*}
where $\tau_i(x)$ is the neighbour of $x$ having different $i$-th coordinate, $\left(\tau_i(x)\right)_i = 1-x_i$. Geometrically speaking, $I_i(A)$  measures the size  of the edge boundary of $A$ in the $i$-th direction.
  A seminal result in the theory of influences is the KKL theorem (by Kahn, Kalai and Linial  \cite{KKL88}). Based on the hypercontractivity inequality, it ensures the 
 existence of a coordinate with a large influence for non-constant boolean functions.

 \medskip

 Recent papers by Keller \cite{kell11ivbf} and Keller, Mossel and Sen \cite{kellms12gi} develop the theory of influences in the case of a continuous space. They propose two different definitions: $h$-influences \cite{kell11ivbf} involve the measures of the intersections of a given set with all lines in the $i$-th canonical direction, while  geometric influences \cite{kellms12gi} involve  the boundary measures of the intersections with lines in the $i$-th direction.

\begin{defi}
Let $n \in \N^*$, $i \in \left\{1,...,n\right\}$, $x \in \R^n$ and $A$ a Borel subset of $\R^n$.
  For  $z \in \R^{n-1}$, we set  $$A^z_i = \left\{y \in \R \,\Big| \, \left(z_1,...,z_{i-1},y,z_i,...,z_{n-1}\right) \in A\right\}. $$ 
   Let  $\nu = \nu_1 \otimes ... \otimes \nu_n$ be a product probability measure on $\R^n$. 
  
If $h : \left[0,1\right] \rightarrow \R_+$ is a measurable function, the $h$-influence of the $i$-th coordinate on $A$ with respect to $\nu$ is defined by
\begin{equation*}
\mathcal{I}^h_{\nu,i}(A) = \int_{\R^{n-1}}h\left(\nu_i\left(A^z_i\right)\right)d\widehat{\nu}^i(z),
\end{equation*}
where $\widehat{\nu}^i = \nu_1\otimes...\otimes\nu_{i-1}\otimes\nu_{i+1}\otimes...\otimes\nu_n$.

The geometric influence of the $i$-th coordinate on $A$ with respect to the measure $\nu$ is given by
\begin{equation*}
\mathcal{I}^{\mathcal{G}}_{\nu,i}(A) = \int_{\R^{n-1}}\left(\nu_i\right)^+\left(A^z_i\right)d\widehat{\nu}^i(z).
\end{equation*}
When the choice of the underlying measure is obvious, we simply write  $\mathcal{I}^h_i(A)$ and
$\mathcal{I}^{\mathcal{G}}_i(A)$.
\end{defi}

Keller was able to prove an analogue of the KKL theorem for $h$-influences provided $h$ 
is larger than the entropy function  $\mathrm{Ent}$  defined by  $\mathrm{Ent}(x) = x\log\frac1x + (1-x)\log\frac1{1-x}$ for $x\in(0,1)$ and  $\mathrm{Ent}(0)=\mathrm{Ent}(1)=0$. His result \cite{kell11ivbf} is stated
for  functions on the unit cube, equipped with Lebesgue's measure. 
Using a standard transportation argument yields the following formulation:
\begin{theo}\label{th:kell}
Let $\mu$ be a probability measure on $\R$. Then, for every Borel set $A \subseteq \R^n$
\begin{equation*}
\max_{1\le i\le n}\mathcal{I}^{\mathrm{Ent}}_i(A) \geq \gamma \, \mu^n(A)\mu^n\left(A^c\right)\frac{\log n}n,
\end{equation*}
where $\gamma > 0$ is a universal constant.
\end{theo}

Keller, Mossel and Sen \cite{kellms12gi} establish an analogue of the KKL theorem for geometric
influences for Boltzmann measures  $d\mu_\rho(t)=\exp(-|t|^\rho) dt/Z_\rho\,  dt$ with $\rho\ge 1$
(and under mild assumptions for log-concave measures enjoying the same isoperimetric inequality
as $\mu_\rho$). Thanks to Theorem~\ref{th:dir} we can propose a more general result:

\begin{theo}\label{th:GI}
Let $\mu$ be an even log-concave probability measure on $\R$, with positive and $C^1$-bounded   density $\varphi_{\mu}$. 
Assume that  $I_\mu\ge J$ where $J$ is a non-negative function on $[0,1]$, which is symmetric with respect to $1/2$, verifies $J(0)=0$
and $t\mapsto J(t)/(t\log(1/t))$ is essentially non-decreasing on $(0,1)$ with constant $D$.
Then for every Borel set $A\subset \R^n$, 
   $$\max_{1\le i\le n} \mathcal I_i^{\mathcal G}(A)\ge \alpha_D \mu^n(A)\mu^n(A^c) J\left(\frac{1}{n}\right),$$
  where $\alpha_D \geq \frac{\kappa}{D^3}$ and  $\kappa > 0$ is a  universal constant.
\end{theo}

\begin{rem}\label{rem:final}
Actually the conclusion of Theorem \ref{th:GI} holds under less restrictive conditions on the measure. In particular the log-concavity assumption can be removed, either by using a different
symmetrization argument than in \cite{kellms12gi} or by introducing another definition of 
the geometric influence based on notions of geometric measure theory. These modifications
require a substantial and technical work that will appear in the PhD dissertation of the second-named author. In the present paper, we simply explain how the argument of Keller, Mossel and
Sen can be adapted, putting forward the parts of the reasoning where the conditions on  $J$ are used.

\end{rem}

The next lemma follows from Proposition 1.3 in \cite{kellms12gi}.  It explains the connection between geometric influences and boundary measure for the uniform enlargement:
\begin{lem}\label{lem:eqmono}
Let $\mu$ be  as in Theorem~\ref{th:GI}. Let $A \subseteq \R^n$ be a monotone increasing set (in the following sense: if $x\in A$ and for all $i$, $x_i\le y_i$, then $y\in A$). Then
\begin{equation*}
\sum^{n}_{i=1}\mathcal{I}^{\mathcal{G}}_i(A) = \left(\mu^n\right)^+(A).
\end{equation*}
\end{lem}

\begin{proof}[Proof of Theorem \ref{th:GI}]
We follow the argument of Keller, Mossel and Sen. Assume $n\ge 2$. 
Let $A \subseteq \R^n$ be as in the statement of the theorem. Lemma 3.7 of \cite{kellms12gi} ensures that, without loss of generality, one can assume that  $A$ is increasing. 
Set $t=\mu^n(A)$. Since $A$ and $A^c$ have the same influences, we may assume that $t\le 1/2$
($A^c$ is monotone decreasing, but passing to its image by the symmetry with respect to the origin
we may ensure that we work is an increasing set of measure at most $1/2$).
We distinguish two cases :

\textit{First case :} $t \leq 1/n$.
Thanks to Lemma \ref{lem:eqmono} and to the isoperimetric inequality of Theorem \ref{th:dir}
\begin{equation*}
\sum^n_{i=1}\mathcal{I}^{\mathcal{G}}_i(A) = \left(\mu^n\right)^+(A) \geq \frac{1}{c_D} J\left(\mu^n(A)\right)= \frac{1}{c_D} J(t),
\end{equation*}
with $c_D = \frac{2D^2}{(\log 2)^2}$. Lemma \ref{lem:equiv} asserts that $s\mapsto J(s)/s$
is essentially non-increasing on $(0,1/2]$ with constant $D/\log 2$, therefore $J(t)/t\ge  n J(1/n)\log2/D$. Consequently

\begin{equation*}
\max_i \mathcal{I}^{\mathcal{G}}_i(A) \geq \frac1n\sum^n_{i=1}\mathcal{I}^{\mathcal{G}}_i(A) \geq 
\frac{\log 2}{D c_D} tJ\left(\frac1n\right)\ge 
\frac{\left(\log 2\right)^3}{2D^3}t(1-t)J\left(\frac1n\right).
\end{equation*}

\textit{Second  case :} $t\in(1/n,1/2]$. 
The argument uses three main ingredients. The first one is the isoperimetric inequality $I_\mu\ge J$, which implies that for all $i\le n$,
\begin{equation}\label{eq:ingredient1}
\mathcal{I}^{\mathcal{G}}_i(A)
= \int_{\R^{n-1}}\mu^+\left(A^z_i\right)d\mu^{n-1}(z)
\ge \int_{\R^{n-1}}J\left(\mu\left(A^z_i\right)\right)d\mu^{n-1}(z)
= \mathcal{I}^{J}_i(A).
\end{equation}
The second ingredient is a comparison between $J$-influences and $\mathrm{Ent}$-influences:
observe that $\mathrm{Ent}$ is symmetric with respect to $1/2$ and is increasing and one-to-one
on $[0,1/2]$. Let $\mathrm{Ent}^{-1}:[0,\log 2]\to  [0,1/2]$ be its reciprocal function. Then for all $i\le n$,
\begin{equation}\label{eq:ingredient2}
\mathcal{I}^{J}_i(A) \ge \frac{J(s)}{2D} \quad \mathrm{with} \quad  s:=\mathrm{Ent}^{-1} \left( 
\frac{\mathcal{I}^{\mathrm{Ent}}_i(A)}{2}\right) \in \left[0,\frac12\right].
\end{equation}
We postpone the proof of this inequality, and explain how to conclude. The last ingredient is Keller's version of the KKL inequality (Theorem \ref{th:kell}).  It provides  an index
$i$ such that $\mathcal{I}^{\mathrm{Ent}}_i(A)\ge \gamma t(1-t) \log(n)/n$, where $\gamma>0$ is a universal constant. Observe for further use that necessarily $\gamma<8$ (indeed, $\mathrm Ent\le \log 2$ and one can choose $n=2$ and $t=1/2$ in Keller's theorem). Let us show that 
the $i$-th coordinate has a large geometric influence.

Observe that for every $y \in \left[0,\frac12\right]\subset[0,\log 2]$,
$\theta(y) := \frac y{2\log\frac1y} \leq \mathrm{Ent}^{-1}(y)$.
Since $\mathcal{I}^{\mathrm{Ent}}_i(A)/2\le \log(2)/2\le 1/2$, and $\theta$ is increasing on $(0,1)$,
\begin{eqnarray*}
s& = &\mathrm{Ent}^{-1}\left(\frac{\mathcal{I}^{\mathrm{Ent}}_i(A)}{2}\right) \geq \theta\left(\frac{\mathcal{I}^{\mathrm{Ent}}_i(A)}{2}\right) \geq \theta\left(\frac \gamma2t(1-t)\frac{\log n}n\right)\\
 &=& \frac{\gamma t(1-t)}{4n}\frac{\log n}{\log\frac{2n}{\gamma t(1-t)\log n}} \geq
   \frac{\gamma t(1-t)}{4n}\times\frac{\log n}{\log\frac{4n^2}{\gamma \log n}}, 
\end{eqnarray*}
where the latter inequality relies on $t(1-t)\ge (1-t)/n\ge 1/(2n)$.
Since $\gamma\le 8$, the last fraction in the lower bound of $s$ is a positive function of $n\ge 2$ with a positive limit when $n$ tends to infinity. Hence there exists $c=c(\gamma)>0$ such that 
$s\ge ct(1-t)/n$. 
It remains to combine this estimate with $\mathcal{I}^{\mathcal{G}}_i(A) \geq J(s)/(2D)$, a consequence of \eqref{eq:ingredient1} and \eqref{eq:ingredient2}:

If $s\leq \frac1n$, then we also use Lemma \ref{lem:equiv}, which asserts  that $u\mapsto J(u)/u$ is essentially non-increasing on $\left(0,\frac12\right]$ with constant $\frac D{\log 2}$: 
\begin{equation*}
\mathcal{I}^{\mathcal{G}}_i(A) \geq \frac{J(s)}{2D} \geq \frac{\log 2}{2D^2}\, s\,\frac{ J\left(\frac1n\right)}{\frac1n} \geq \frac{\log 2 }{2D^2}\, ct(1-t)J\left(\frac1n\right).
\end{equation*}

 If $s > \frac1n$ then, using the fact that $J$ is essentially non-decreasing on $\left(0,\frac12\right)$ with constant $\frac{2D}{e\log 2}$ (see Lemma \ref{lem:tech}), we get 
\begin{equation*}
\mathcal{I}^{\mathcal{G}}_i(A) \geq \frac{J(s)}{2D} \geq \frac{e\log 2}{4D^2}J\left(\frac1n\right) \geq \frac{e\log 2}{D^2} t(1-t)J\left(\frac1n\right).
\end{equation*}

Eventually, we give a proof for \eqref{eq:ingredient2}. By hypothesis, $J/J_1$ is essentially non-decreasing with constant $D$, where $J_1(x)=x\log(1/x)$. 
 Observe that for $x\in \left[0,\frac12\right]$,  
 $$J_1(x) \leq \mathrm{Ent}(x) = J_1(x) + J_1(1-x) \leq 2J_1(x).$$ It follows that $\frac {J}{\mathrm{Ent}}$ is essentially non-decreasing on $\left(0,\frac12\right]$ with constant $2D$. 
Recall that 
$s \in \left(0,\frac12\right)$ verifies 
$\mathrm{Ent}\left(s\right) = \mathcal{I}^{\mathrm{Ent}}_i(A)/2$. Note  that, if $x \notin \left[s,1-s\right]$, then $\mathrm{Ent}(x) < \mathcal{I}^{\mathrm{Ent}}_i(A)/2$. This yields
\begin{align*}
\int_{\mu\left(A^z_i\right)\in \left[s,1-s\right]}\mathrm{Ent}\left(\mu\left(A^z_i\right)\right)d\mu^{n-1}(z)
& = \mathcal{I}^{\mathrm{Ent}}_i(A) - \int_{\mu\left(A^z_i\right)\notin \left[s,1-s\right]}\mathrm{Ent}\left(\mu\left(A^z_i\right)\right)d\mu^{n-1}(z)
 \geq \frac{\mathcal{I}^{\mathrm{Ent}}_i(A)}2.
\end{align*}
Therefore, using in addition the symmetry with respect to $1/2$ of $J$ and $\mathrm{Ent}$ and
the fact that $\frac J{\mathrm{Ent}}$ is essentially non-decreasing on  $(0,1/2]$ with constant $2D$, we get
\begin{align*}
\mathcal{I}^{J}_i(A)
& \geq \int_{\mu\left(A^z_i\right)\in \left[s,1-s\right]}J\left(\mu\left(A^z_i\right)\right)d\mu^{n-1}(z) \\
& = \int_{\mu\left(A^z_i\right)\in \left[s,1-s\right]}J\big(\min(\mu\left(A^z_i\right),1-\mu\left(A^z_i\right))\big)d\mu^{n-1}(z) \\
& \geq \frac1{2D} \frac{J(s)}{\mathrm{Ent}(s)} \int_{\mu(A^z_i)\in [s,1-s]} \mathrm{Ent}\big(\min(\mu(A^z_i),1-\mu(A^z_i)) \big)d\mu^{n-1}(z) \\
& =\frac1{2D}\frac{J(s)}{\mathrm{Ent}(s)}\int_{\mu(A^z_i)\in [s,1-s]}\mathrm{Ent}\big(\mu(A^z_i)\big) d\mu^{n-1}(z) \\
 & \geq  \frac1{2D}\frac{J\left(s\right)}{\mathrm{Ent}\left(s\right)}\frac{\mathcal{I}^{\mathrm{Ent}} _i(A)}2 = \frac{J\left(s\right)}{2D}.
\end{align*}
The proof is  complete.

\end{proof}


\begin{thebibliography}{10}

\bibitem{AFPBVF00}
Luigi Ambrosio, Nicola Fusco, and Diego Pallara.
\newblock {\em Functions of bounded variation and free discontinuity problems}.
\newblock Oxford Mathematical Monographs. The Clarendon Press, Oxford
  University Press, New York, 2000.

\bibitem{bart02lsmi}
F.~Barthe.
\newblock Log-concave and spherical models in isoperimetry.
\newblock {\em Geom. Funct. Anal.}, 12(1):32--55, 2002.

\bibitem{bart04idii}
F.~Barthe.
\newblock Infinite dimensional isoperimetric inequalities in product spaces
  with the supremum distance.
\newblock {\em J. Theoret. Probab.}, 17(2):293--308, 2004.

\bibitem{bartcr06iibe}
Franck Barthe, Patrick Cattiaux, and Cyril Roberto.
\newblock Interpolated inequalities between exponential and {G}aussian,
  {O}rlicz hypercontractivity and isoperimetry.
\newblock {\em Rev. Mat. Iberoam.}, 22(3):993--1067, 2006.

\bibitem{bobk97ipue}
S.~G. Bobkov.
\newblock Isoperimetric problem for uniform enlargement.
\newblock {\em Studia Math.}, 123(1):81--95, 1997.

\bibitem{bobkh97scbi}
S.~G. Bobkov and C.~Houdr{\'e}.
\newblock {\em Some connections between isoperimetric and {S}obolev-type
  inequalities}, volume 129 of {\em Mem. Amer. Math. Soc.}
\newblock 1997.

\bibitem{bobkh00wdfc}
S.~G. Bobkov and C.~Houdr{\'e}.
\newblock Weak dimension-free concentration of measure.
\newblock {\em Bernoulli}, 6(4):621--632, 2000.

\bibitem{bolll91eiig}
B.~Bollob{\'a}s and I.~Leader.
\newblock Edge-isoperimetric inequalities in the grid.
\newblock {\em Combinatorica}, 11:299--314, 1991.

\bibitem{KKL88}
Kalai~G. Kahn, J. and N.~Linial.
\newblock The influence of variables on boolean functions.
\newblock In {\em Proceedings of 29th IEEE Symposium on Foundations of Computer
  Sciences}, pages 68--80, 1988.

\bibitem{kalsaf06}
Gil Kalai and Shmuel Safra.
\newblock Threshold phenomena and influence: perspectives from mathematics,
  computer science, and economics.
\newblock In {\em Computational complexity and statistical physics}, St. Fe
  Inst. Stud. Sci. Complex., pages 25--60. Oxford Univ. Press, New York, 2006.

\bibitem{kell11ivbf}
Nathan Keller.
\newblock On the influences of variables on {B}oolean functions in product
  spaces.
\newblock {\em Combin. Probab. Comput.}, 20(1):83--102, 2011.

\bibitem{kellms12gi}
Nathan Keller, Elchanan Mossel, and Arnab Sen.
\newblock Geometric influences.
\newblock {\em Ann. Probab.}, 40(3):1135--1166, 2012.

\bibitem{kuc09itfei}
Marek Kuczma.
\newblock {\em An introduction to the theory of functional equations and
  inequalities}.
\newblock Birkh\"auser Verlag, Basel, second edition, 2009.
\newblock Cauchy's equation and Jensen's inequality, Edited and with a preface
  by Attila Gil{\'a}nyi.

\bibitem{latao00bsp}
R.~Lata{\l}a and K.~Oleszkiewicz.
\newblock Between {S}obolev and {P}oincar\'e.
\newblock In {\em Geometric aspects of functional analysis}, number 1745 in
  Lecture Notes in Math., pages 147--168, Berlin, 2000. Springer.

\bibitem{milm09rcfi}
Emanuel Milman.
\newblock On the role of convexity in functional and isoperimetric
  inequalities.
\newblock {\em Proc. Lond. Math. Soc. (3)}, 99(1):32--66, 2009.

\bibitem{morg06iep}
Frank Morgan.
\newblock Isoperimetric estimates in products.
\newblock {\em Ann. Global Anal. Geom.}, 30(1):73--79, 2006.

\bibitem{peet68if2}
J.~Peetre.
\newblock On interpolation functions. {II}.
\newblock {\em Acta Sci. MAth. (Szeged)}, 29:91--92, 1968.

\bibitem{ros01ip}
A.~Ros.
\newblock The isoperimetric problem.
\newblock In {\em Global theory of minimal surfaces}, volume~2 of {\em Clay
  Math. Proc.}, pages 175--209. Amer. Math. Soc., Providence, RI, 2005.

\bibitem{tala91niic}
M.~Talagrand.
\newblock A new isoperimetric inequality and the concentration of measure
  phenomenon.
\newblock In J.~Lindenstrauss and V.~D. Milman, editors, {\em Geometric Aspects
  of Functional Analysis}, number 1469 in Lecture Notes in Math., pages
  94--124, Berlin, 1991. Springer-Verlag.

\end{thebibliography}

\bigskip\noindent
F. Barthe, B. Huou: Institut de Math\'ematiques de Toulouse ; UMR5219. Universit\'e de Toulouse ; CNRS.  UPS IMT, F-31062 Toulouse Cedex 9, France

\end{document}